\documentclass[10pt]{amsart}
\usepackage{amsmath}
\usepackage{xcolor}
\usepackage{tikz}
\usepackage{setspace}
\usepackage{tikz-cd}

\usetikzlibrary{arrows,decorations.pathmorphing,backgrounds,positioning,fit}
\newtheorem{theorem}{Theorem}[section]
\newtheorem{lemma}{Lemma}[section]
\newtheorem{defn}{Definition}[section]
\newtheorem{ex}{Example}[section]
\newtheorem{remark}{Remark}[section]

\newcommand{\CC}{\mathbb{C}}
\newcommand{\RR}{\mathbb{R}}
\newcommand{\NN}{\mathbb{N}}
\newcommand{\QQ}{\mathbb{Q}}

\definecolor{jade}{rgb}{0.10, 0.56, 0.42}
\definecolor{cerise}{rgb}{0.87, 0.19, 0.39}

\tikzstyle{mutable}=[inner sep=0.5mm,circle,draw,minimum size=2mm]
\tikzstyle{frozen}=[inner sep=1mm,rectangle,draw]
\tikzstyle{outline}=[thick,line width=1.5mm,draw=black!10]

\title[Dimer models on cylinders over Dynkin diagrams]
{Dimer models on cylinders over Dynkin diagrams and cluster algebras}

\author{Maitreyee C. Kulkarni}
\address{Department of Mathematics,
Louisiana State University, Baton Rouge, LA, USA}
\email{mkulka2@lsu.edu}

\thanks{The author was supported by the NSF grant DMS-1601862, and an LSU Dissertation Year Fellowship.}

\begin{document}
\begin{abstract} 
In this paper, we describe a general setting for dimer models on cylinders over Dynkin diagrams which in type A reduces to the well-studied case of dimer models on a disc.  We prove that all Berenstein--Fomin--Zelevinsky quivers for Schubert cells in a symmetric Kac--Moody algebra give rise to  dimer models on the cylinder over the corresponding Dynkin diagram.  We also give an independent proof of a result of Buan, Iyama, Reiten and Smith that the corresponding superpotentials are rigid using the dimer model structure of the quivers. 
\end{abstract}
\maketitle

\section{Introduction} 
Cluster algebras were defined by Fomin and Zelevinsky in 2000 \cite{fz02} to study Lusztig's dual canonical basis of quantum groups. A cluster algebra is a certain commutative ring that lies somewhere between a polynomial ring and its field of fractions, and it is generated from an initial collection of data (a quiver and a function on each vertex) by a combinatorial procedure called mutation.   In particular, a cluster algebra is defined starting from a quiver (or directed graph) with $n$ vertices where each vertex $i$ has a function $x_i$ on it.  A process called mutation changes both the quiver and the functions on the vertices, and iteratively produces the generating set of the cluster algebra.

Let $G$ be a Lie group of type ADE and $P$ be a parabolic subgroup. In this setting, Geiss, Leclerc and Schr\"{o}er gave a cluster structure on the coordinate ring of the partial flag variety $G/P$ \cite{gls}. This gives a categorification of the coordinate ring of the affine open cell in $G/P$ by a subcategory of modules over the preprojective algebra associated to the Dynkin diagram of G. This categorification is then lifted to the homogeneous coordinate ring. Jensen, King and Su gave a direct categorification of this homogeneous coordinate ring for Grassmannians, i.e. when $G$ is of type A and $P$ is a maximal parabolic subgroup \cite{jks16}. This is done using the category of (maximal) Cohen--Macaulay modules $T$ over $B$, where $B$ is a quotient algebra of a certain preprojective algebra. 

Recently, Baur, King and Marsh gave a combinatorial model for this categorification.  They used Postnikov diagrams, which were used by Scott to show that the homogeneous coordinate ring of $\mathrm{Gr(k,n)}$ is a cluster algebra \cite{s06}. A Postnikov diagram encodes information about seeds of the cluster algebra and its clusters.  Each region in a Postnikov diagram is labelled by a $k$-subset of $\{1, 2, \ldots, n\}$. Let $I$ be a $k$-subset of $\{1, 2, \ldots, n\}$ corresponding to a minor of the matrix and $M_I$ be a certain Cohen--Macaulay $B$-module associated to $I$. To each Postnikov diagram $D$, associate the module $T_D=\bigoplus_I M_I$.  They define a dimer algebra as the Jacobian algebra for the quiver corresponding to a Postnikov diagram. One of Baur King Marsh's main results is that the dimer algebra $A_D \cong \mathrm{End}_B(T_D)$ \cite{bkm16}, which gives a combinatorial construction of the endomorphism algebra required for their categorification.

The key idea in this paper is to realize cluster algebras associated to symmetric Kac--Moody algebras  by Berenstein--Fomin--Zelevinsky quivers defined in \cite{bfz05}.  We will introduce conceptual framework for dimer models in other types.  The dimer models will play the role of quivers from Postnikov diagrams.  Let $G$ be a Kac--Moody group and $W$ be its Weyl group.  For any pair $(u, v) \in W \times W$, associate a quiver $Q^{u, v}$ following \cite{bfz05}.  In type A, $Q^{u,v}$ is planar for any $u,v \in W$ but in other types, these quivers are not planar in general.  


In this paper, we will see that these quivers can be realized as dimer models on the cylinder over the Dynkin diagram of $G$. Suppose $\Gamma$ is the Dynkin diagram corresponding to  $G$, then $\Gamma \times \RR$ is called the cylinder over the Dynkin diagram $\Gamma$.  A vertex in the Dynkin diagram is called a branching point if it has more than two edges incident to it.  A vertex is called an endpoint if it has exactly one edge incident to it.  Let $V$ be the set of endpoints and branching points of a Dynkin diagram.  The path $\Gamma_{m,n}$ between any two vertices $m$ and $n$ in $V$ is called a branch in the Dynkin diagram. The space $\Gamma_{m,n}\times \RR$ is called the sheet of the cylinder over the branch $\Gamma_{m,n}$. If a Dynkin diagram has $k$ branches then the cylinder over the Dynkin diagram has $k$ sheets glued at a string on every branching point. This realization makes the quiver planar in each sheet of the cylinder. 
\begin{theorem} The quiver $Q^{u,e}$ corresponding to any pair $(u, e)$, where $e$ is the identity element and $u$ is an arbitrary element in the Weyl group, has the following structure: 
\begin{itemize}
\item Each face of $Q^{u,e}$ is oriented.
\item Each face of $Q^{u,e}$ on the cylinder $\Gamma \times \RR$ projects onto an edge of the Dynkin diagram.
\item Each edge of $Q^{u,e}$ projects onto a vertex of the Dynkin diagram or an edge of the Dynkin diagram. 
\end{itemize}
\end{theorem}

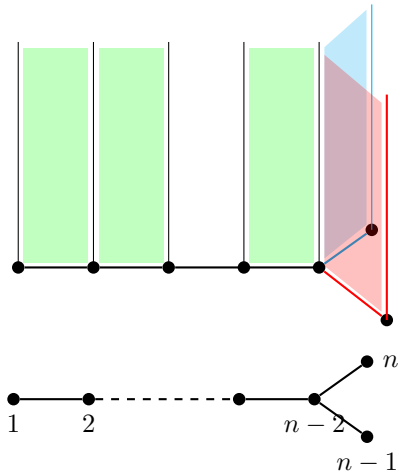
\begin{figure}[ht!]
\begin{tikzpicture}
\node [fill,circle,scale=0.5,label=below:$$] (0) at (-1,0) {};
\node [fill,circle,scale=0.5,label=below:$$] (1) at (0,0) {};
\node [fill,circle,scale=0.5,label=below:$$] (2) at (1,0) {};
\node [fill,circle,scale=0.5,label=below:$$] (3) at (2,0) {};
\node [fill,circle,scale=0.5,label=below:$$] (4) at (3.9,-0.7) {};
\node [fill,circle,scale=0.5,label=right:$$] (n) at (3.7,0.5) {};
\node [fill,circle,scale=0.5,label=below:$$] (3) at (3,0) {};
\draw [thick,-] (0)-- (1);
\draw [thick,-] (1)-- (2);
\draw [thick,-] (2)-- (3,0);
\draw [thick,-,red] (3)-- (4);
\draw [thick,-,cyan] (3)-- (n);
\draw (0,0)-- (0,3);
\draw (-1,0)-- (-1,3);
\draw (1,0)-- (1,3);
\draw (2,0)-- (2,3);
\draw (3,0)-- (3,3);
\draw [cyan](3.7,0.5)-- (3.7,3.5);
\draw [thick,-,red] (3.9,-0.7)-- (3.9,2.3);
\fill[nearly transparent, green]  (-0.93,2.92) rectangle (-0.07,0.07);
\fill[nearly transparent, green] (0.07,0.07) rectangle (0.93,2.92);
\fill[nearly transparent, green]  (2.93,2.92) rectangle (2.07,0.07);
\fill[nearly transparent,cyan] (3.07,0.12) -- (3.07,2.93)--(3.63,3.43)--(3.63,0.57);
\fill[nearly transparent,red] (3.07,0.02) -- (3.07,2.83)--(3.83,2.19)--(3.83,-0.59);
\end{tikzpicture}

\begin{tikzpicture}
\node [fill,circle,scale=0.5,label=below:$1$] (1) at (0,0) {};
\node [fill,circle,scale=0.5,label=below:$2$] (2) at (1,0) {};
\node [fill,circle,scale=0.5,label=below:$n-1$] (4) at (4.7,-0.5) {};
\node [fill,circle,scale=0.5,label=right:$n$] (n) at (4.7,0.5) {};
\node [fill,circle,scale=0.5,label=below:$$] (3) at (3,0) {};
\node [fill,circle,scale=0.5,label=below:$n-2$] (5) at (4,0) {};
\draw [thick,-] (1)-- (2);
\draw [thick,-,dashed] (2)-- (3,0);
\draw [thick,-] (5)-- (4);
\draw [thick,-] (5)-- (n);
\draw [thick,-] (3)-- (5);
\end{tikzpicture}
\caption{The cylinder over Dynkin diagram of type $D_n$}
\end{figure}

Any quiver on a cylinder over a Dynkin diagram that has the above properties will be called a dimer model on the cylinder [see Definition 3.2] because it will play a similar role to the dimer models of \cite{bkm16}.  The dimer algebra of \cite{bkm16} will be replaced  by the Jacobian algebra (see  \cite{dwz08}) corresponding to a certain potential of the BFZ quiver.  The Jacobian algebra $J(Q, S)$ of the quiver $Q$ depends on the choice of a potential.  In this case we use a particularly nice type of potential called a rigid potential.  Every rigid potential is non-degenerate which means that any sequence of mutations of the quiver with potential does not create a 2-cycle in the quiver. 

We define the superpotential $S$ of a quiver $Q$ as follows: \[S= \sum{\text{clockwise oriented faces}}-\sum{\text{anti-clockwise oriented faces}}.\]  Note that a face of a quiver is a cycle which is not divided by an edge.  We will show that this is a rigid potential, i.e. that all cycles in the quiver $Q$ lie in the Jacobian ideal of the potential $S$.  This will be proved in two steps, first for faces, and then for non-self-intersecting oriented cycles. As each cycle in the quiver is oriented, the above two cases cover all cycles in the quiver. 

The dimer model structure of these quivers reduces the global problem of verification of rigidity of the super-potential to a local problem on each sheet.  Let $S_r$ be the superpotential of the subquiver $Q_r$ of $Q$ drawn on the $rth$ sheet of the cylinder.  We show that $S_r$ is rigid in the $rth$ sheet, for each $r$.  As sheets are glued at a string, they only share edges that lie on the gluing string with each other.  Particularly, they do not share any faces, therefore the superpotential $S$ is simply the sum $S=\sum_r S_r$. As each face belongs to a unique sheet, the gluing of sheets does not affect the rigidity of the potential. 

In Section 2, we give some preliminary definitions.  In Section 3, we define a Berenstein--Fomin--Zelevinsky quiver, then give our construction of cylinders on Dynkin diagrams and quivers from double Bruhat cells. In this section we prove that the BFZ quivers can be realized as dimer models on the cylinder over a Dynkin diagram.  The quiver lies entirely on the cylinder by construction.  In the last section, as an application, we give an independent proof of the following result in  \cite{birs11}:
\begin{theorem}  For any Weyl group element $u \in W$, the superpotential of the BFZ quiver $Q^{u,e}$ is rigid. 
\end{theorem}
We prove that the quiver is planar in each sheet of the cylinder and each face of the quiver is oriented.   We prove the rigidity in each sheet by observing that the faces on the boundary belong to the  Jacobian ideal.  Then we use induction on the faces of dimer models to prove that every face belongs to the ideal.  Then we notice that every cycle can be written in terms of faces that the cycle contains, which tells us that each cycle is in the ideal. 

{\bf Acknowledgements.~}
I am grateful to my advisor Milen Yakimov for his advice on this project,  many helpful discussions and great ideas. I would also like to thank Robert Marsh for his insightful discussions on dimer models.

\section{Preliminaries}
Cluster algebras are defined using quivers and their mutations.   A quiver is a directed graph. We denote it by $Q$, its set of vertices by $Q_0$ and its set of edges by $Q_1$.  Consider two  maps $s, t: Q_1 \to Q_0$. The map $s(\alpha)=$ the starting vertex of the edge $\alpha$ and $t(\alpha)=$ the end-vertex of the edge $\alpha$.
\begin{defn}
A \textbf{seed} is a quiver together with elements $\{x_i\}_{i \in Q_0}$ of a field on the vertices that together freely generate that field over $\mathbb{Q}$.  The elements on the vertices are called \textbf{cluster variables}.
\end{defn}
\begin{defn} The  process of mutation of a seed at vertex $k$ is defined as follows:
\begin{itemize}
\item Step 1:  Reverse all arrows touching the vertex $x_k$.
\item Step 2: Complete triangles, i.e., for every path $i \to k \to j$, and an edge $j \to i$.
\item Step 3: Cancel any 2-cycles created in Step 2.
\item Step 4: Replace $x_k$ at the vertex $k$ with $\displaystyle{x_k^{\prime}=\frac{ \prod_{k \to l} x_l + \prod_{l \to k} x_l}{x_k}}$ where the products are over edges with source vertex $k$ and with target vertex $k$ respectively.
\end{itemize}
\end{defn}

\begin{defn} Let $Q$ be a finite quiver without loops or 2-cycles with vertices ${1, \cdots , n}$ and the initial seed $(Q, x_1, \ldots , x_n)$. The cluster algebra $\mathcal{A_Q}$ is the subalgebra of $\QQ(x_1,x_2,\ldots, x_n)$ generated by all cluster variables obtained from all possible sequences of mutations applied to the initial seed.
\end{defn}

\begin{ex} We show an example of mutation of the following quiver $Q$ at the vertex labelled $x_2$.

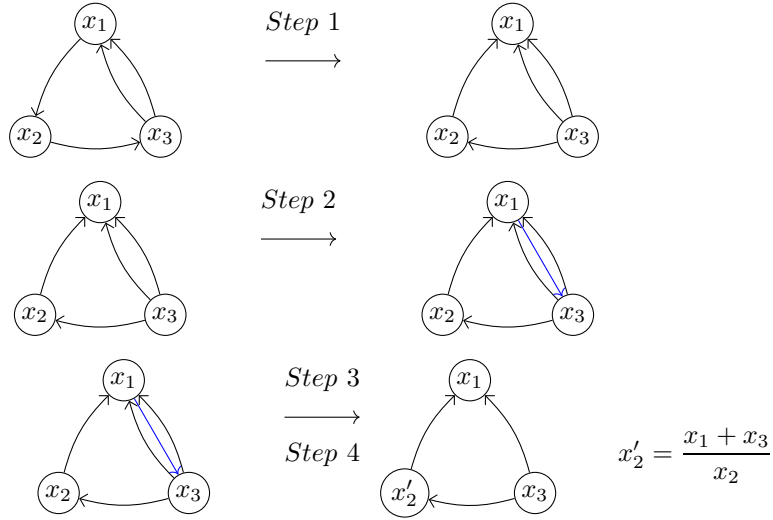
\begin{figure}[ht!]

\begin{minipage}[t]{0.25\textwidth}
\kern0pt
\raggedright
\begin{tikzpicture}
	\begin{scope}[xshift=3in,yshift=-.5cm]
		\node[mutable] (x1) at (90:1) {$x_1$};
		\node[mutable] (x2) at (210:1) {$x_2$};
		\node[mutable] (x3) at (-30:1) {$x_3$};

		\draw[-angle 90,relative, out=-15,in=-165] (x1) to (x2);
	
		\draw[-angle 90,relative, out=-15,in=-165] (x2) to (x3);
		\draw[-angle 90,relative, out=15,in=165] (x3) to (x1);
		\draw[-angle 90,relative, out=-15,in=-165] (x3) to (x1);
    
	\end{scope}
	\end{tikzpicture}
\end{minipage}
\begin{minipage}[t]{0.1\textwidth}
\kern0pt
\raggedright
\begin{tikzpicture}
	\begin{scope}[xshift=3in,yshift=-.5cm]
		\node (x1) at (90:0.5) {$Step \  1$};
		\draw[->](-0.5,0) to (0.5,0);
    \end{scope}
	\end{tikzpicture}
\end{minipage}
\begin{minipage}[t]{0.25\textwidth}
\kern0pt
\raggedleft
\begin{tikzpicture}
	\begin{scope}[xshift=3in,yshift=-.5cm]
		\node[mutable] (x1) at (90:1) {$x_1$};
		\node[mutable] (x2) at (210:1) {$x_2$};
		\node[mutable] (x3) at (-30:1) {$x_3$};

		\draw[-angle 90,relative, out=15,in=165] (x2) to (x1);
	
		\draw[-angle 90,relative, out=15,in=165] (x3) to (x2);
		\draw[-angle 90,relative, out=15,in=165] (x3) to (x1);
		\draw[-angle 90,relative, out=-15,in=-165] (x3) to (x1);
    
	\end{scope}
	\end{tikzpicture}
\end{minipage}

\vspace{2mm}

\begin{center}
\begin{minipage}[t]{0.25\textwidth}
\kern0pt
\raggedright
\begin{tikzpicture}
	\begin{scope}[xshift=3in,yshift=-.5cm]
		\node[mutable] (x1) at (90:1) {$x_1$};
		\node[mutable] (x2) at (210:1) {$x_2$};
		\node[mutable] (x3) at (-30:1) {$x_3$};
 		\draw[-angle 90,relative, out=15,in=165] (x2) to (x1);
		\draw[-angle 90,relative, out=15,in=165] (x3) to (x2);
		\draw[-angle 90,relative, out=15,in=165] (x3) to (x1);
		\draw[-angle 90,relative, out=-15,in=-165] (x3) to (x1);
    
	\end{scope}
	\end{tikzpicture}

\end{minipage}%
\begin{minipage}[t]{0.1\textwidth}
\kern0pt
\raggedright
\begin{tikzpicture}
	\begin{scope}[xshift=3in,yshift=-.5cm]
		\node (x1) at (90:0.5) {$Step \  2$};
		\draw[->, label=above:$Step 1$](-0.5,0) to (0.5,0);
    \end{scope}
	\end{tikzpicture}
\end{minipage}
\begin{minipage}[t]{0.25\textwidth}
\kern0pt
\raggedleft
\begin{tikzpicture}
	\begin{scope}[xshift=3in,yshift=-.5cm]
		\node[mutable] (x1) at (90:1) {$x_1$};
		\node[mutable] (x2) at (210:1) {$x_2$};
		\node[mutable] (x3) at (-30:1) {$x_3$};

		\draw[-angle 90,relative, out=15,in=165] (x2) to (x1);
	
		\draw[-angle 90,relative, out=15,in=165] (x3) to (x2);
		\draw[-angle 90,relative, out=15,in=165] (x3) to (x1);
		\draw[-angle 90,relative, out=-15,in=-165] (x3) to (x1);
    		\draw[->, blue] (x1) to (x3);
	\end{scope}
	\end{tikzpicture}
\end{minipage}
\end{center}

\vspace{2mm}

\begin{minipage}[t]{0.25\textwidth}
\kern0pt
\raggedright
\begin{tikzpicture}
	\begin{scope}[xshift=3in,yshift=-.5cm]
		\node[mutable] (x1) at (90:1) {$x_1$};
		\node[mutable] (x2) at (210:1) {$x_2$};
		\node[mutable] (x3) at (-30:1) {$x_3$};

		\draw[-angle 90,relative, out=15,in=165] (x2) to (x1);
	
		\draw[-angle 90,relative, out=15,in=165] (x3) to (x2);
		\draw[-angle 90,relative, out=15,in=165] (x3) to (x1);
		\draw[-angle 90,relative, out=-15,in=-165] (x3) to (x1);
		\draw[->, blue] (x1) to (x3);
	\end{scope}
	\end{tikzpicture}

\end{minipage}%
\begin{minipage}[t]{0.1\textwidth}
\kern0pt
\raggedright
\begin{tikzpicture}
	\begin{scope}[xshift=3in,yshift=-.5cm]
		\node (x1) at (90:0.5) {$Step \  3$};
		\node () at (-90:0.5) {$Step \  4$};
		\draw[->](-0.5,0) to (0.5,0);
    \end{scope}
	\end{tikzpicture}
\end{minipage}
\begin{minipage}[t]{0.2\textwidth}
\kern0pt
\raggedleft
\begin{tikzpicture}
	\begin{scope}[xshift=3in,yshift=-.5cm]
		\node[mutable] (x1) at (90:1) {$x_1$};
		\node[mutable] (x2) at (210:1) {$x'_2$};
		\node[mutable] (x3) at (-30:1) {$x_3$};
		\draw[-angle 90,relative, out=15,in=165] (x2) to (x1);
		\draw[-angle 90,relative, out=15,in=165] (x3) to (x2);
		\draw[-angle 90,relative, out=-15,in=-165] (x3) to (x1);
		\node () at (3,0) {$x'_2=\displaystyle\frac{x_1+x_3}{x_2}$};
	\end{scope}
	\end{tikzpicture}
\end{minipage}
\caption{Mutation of a quiver}
\end{figure}
\end{ex}

 \begin{defn}
 A quiver is said to be 2-acyclic if it has no oriented 2-cycles.
\end{defn}
The \textbf{path algebra} $\CC(Q)$ of a quiver $Q$ is an algebra generated by paths in the quiver $Q$ with multiplication given by concatenation of paths. A \textbf{potential} $S \in \CC(Q)$  is a linear combination of cycles in the quiver. The pair $(Q,S)$ of a quiver and its potential is called a quiver with potential or a QP. We will follow the definition of  mutation of quivers with potential in \cite{dwz08}.  One of the hardships of the theory is that mutations of potentials may not produce a 2-acyclic quiver in general. 

\begin{defn} A QP $(Q,S)$ is called non-degenerate if every sequence of mutations of $(Q, S)$ is 2-acyclic.
\end{defn}
The process of verifying non-degeneracy is an infinite process in general, as the quiver may not be mutation finite.  To verify non-degeneracy of a potential without going through this infinite process, we use a stronger condition on a potential called rigidity. 
\begin{theorem} [\cite{dwz08}] Every rigid potential is non-degenerate. 
\end{theorem}
\noindent To define a rigid potential, let us first define the cyclic derivative of a potential.  For every $a \in Q_1$, the cyclic derivative $\partial_a$ is defined as: 
\[ \partial_a(e_1e_2\cdots e_n)= e_{i+1}\cdots e_ne_1\cdots e_{i-1},\]  where $e_1e_2\cdots e_n$ is a cycle in the quiver and  $a=e_i$.  If $a \neq e_i$ for any $i$, then $\partial_a(e_1e_2\cdots e_n)=  0$.
If $S$ is a potential of $Q$, we define the Jacobian ideal $J(S)$ to be the ideal generated by $\partial_a(S)$, for all $a \in Q$. The Jacobian algebra $P(Q,S)$ is the quotient $\CC(Q)/J(S)$. 

\begin{defn} A QP $(Q,S)$ is rigid if every cycle Q is cyclically equivalent  to an element of $J(S)$.
\end{defn}
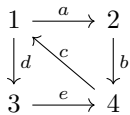
\begin{figure}[ht!]

\begin{tikzcd}[ampersand replacement=\&]
1 \ar{r}{a} \ar{d}{d} \& 2 \ar{d}{b} \& \& \\
3 \ar{r}{e} \& 4 \ar{ul}[above]{c}
\end{tikzcd}
\caption{Rigid and non-rigid potentials}
\label{fig:rigid}
\end{figure}

\begin{ex}
Consider the potential $S_1=abc$ in the quiver in ~Figure \ref{fig:rigid}.  Then, by differentiating $S_1$ with respect to the edges $a$, $b$ and $c$, we get that $J(S_1)=\langle bc, ca, ab\rangle$. So $abc \in J(S_1)$ but $cde \notin J(S_1)$. Therefore $S_1$ is not rigid.
If $S_2=abc+cde$, then  $J(S_2)=\langle bc, ca, ab+de, ec, cd\rangle$. So $abc \in J(S_2)$ and $cde \in J(S_2)$. Therefore $S_2$ is rigid.

\end{ex}

\section{Quivers on cylinders over Dynkin diagrams}

\begin{figure}[ht!]
\begin{tikzpicture}
\node [fill,circle,scale=0.5,label=below:$$] (0) at (-1,0) {};
\node [fill,circle,scale=0.5,label=below:$$] (1) at (0,0) {};
\node [fill,circle,scale=0.5,label=below:$$] (2) at (1,0) {};
\node [fill,circle,scale=0.5,label=below:$$] (3) at (2,0) {};
\node [fill,circle,scale=0.5,label=below:$$] (4) at (3.9,-0.7) {};
\node [fill,circle,scale=0.5,label=right:$$] (n) at (3.7,0.5) {};
\node [fill,circle,scale=0.5,label=below:$$] (3) at (3,0) {};
\draw [thick,-] (0)-- (1);
\draw [thick,-] (1)-- (2);
\draw [thick,-] (2)-- (3,0);
\draw [thick,-,red] (3)-- (4);
\draw [thick,-,cyan] (3)-- (n);
\draw (0,0)-- (0,3);
\draw (-1,0)-- (-1,3);
\draw (1,0)-- (1,3);
\draw (2,0)-- (2,3);
\draw (3,0)-- (3,3);
\draw [cyan](3.7,0.5)-- (3.7,3.5);
\draw [thick,-,red] (3.9,-0.7)-- (3.9,2.3);
\fill[nearly transparent, green]  (-0.93,2.92) rectangle (-0.07,0.07);
\fill[nearly transparent, green] (0.07,0.07) rectangle (0.93,2.92);
\fill[nearly transparent, green]  (2.93,2.92) rectangle (2.07,0.07);
\fill[nearly transparent,cyan] (3.07,0.12) -- (3.07,2.93)--(3.63,3.43)--(3.63,0.57);
\fill[nearly transparent,red] (3.07,0.02) -- (3.07,2.83)--(3.83,2.19)--(3.83,-0.59);
\end{tikzpicture}

\begin{tikzpicture}
\node [fill,circle,scale=0.5,label=below:$1$] (1) at (0,0) {};
\node [fill,circle,scale=0.5,label=below:$2$] (2) at (1,0) {};
\node [fill,circle,scale=0.5,label=below:$3$] (3) at (2,0) {};
\node [fill,circle,scale=0.5,label=below:$n-1$] (4) at (4.7,-0.5) {};
\node [fill,circle,scale=0.5,label=right:$n$] (n) at (4.7,0.5) {};
\node [fill,circle,scale=0.5,label=below:$$] (6) at (3,0) {};
\node [fill,circle,scale=0.5,label=below:$n-2$] (5) at (4,0) {};
\draw [thick,-] (1)-- (2);
\draw [thick,-] (2)-- (3);
\draw [thick,-,dashed] (2)-- (3,0);
\draw [thick,-] (5)-- (4);
\draw [thick,-] (5)-- (n);
\draw [thick,-] (6)-- (5);
\end{tikzpicture}
\caption{The cylinder over Dynkin diagram of type $D_n$}
\end{figure}
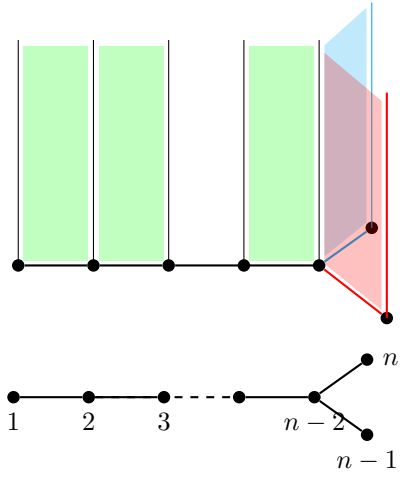

\subsection{Cylinders over Dynkin diagrams} Let $\Gamma$ be a Dynkin diagram. A vertex of a Dynkin diagram is called an endpoint if it has only one edge incident to it. A vertex is called a branching point if it has strictly more than two edges incident to it.  A path $\Gamma_{m, n}$ between two vertices $m$ and $n$ in $\Gamma$ is called a branch if both $m$ and $n$ are branching points or endpoints or if one of them is a branching point and the other is an endpoint. 
\begin{defn} For a Dynkin diagram $\Gamma$, we define the cylinder over $\Gamma$ to be the topological space $\Gamma \times \RR$. Let $\Gamma_0$ be the set of vertices of $\Gamma$. We call the set $\Gamma_0 \times \NN \subset \Gamma \times \RR$ a grid on the cylinder.  The set $\Gamma_{m,n} \times \RR$ is called the sheet over the branch $\Gamma_{m,n}$.  The length of a sheet is the number of edges on the branch.  The subset $\{x_0\} \times \RR$ where $x_0 \in \Gamma_0$ is called a string.
\end{defn}

\begin{ex}  A quiver for double Bruhat cells of $D_n$ can be drawn on a book-like structure as shown in the figure below. The cylinder $D_n \times \RR$ has $n$ strings and three sheets; one sheet of length $n-3$ and two sheets of length 1  glued together at their boundaries. 
\end{ex}

\begin{defn} A quiver on the cylinder over a Dynkin diagram is called a dimer model on the cylinder if
\begin{enumerate} 
\item Each arrow of the quiver projects onto an edge or a vertex of the Dynkin diagram.
\item Each face projects onto an edge of the Dynkin diagram.
\item Each face is oriented. 
\end{enumerate}
\end{defn}

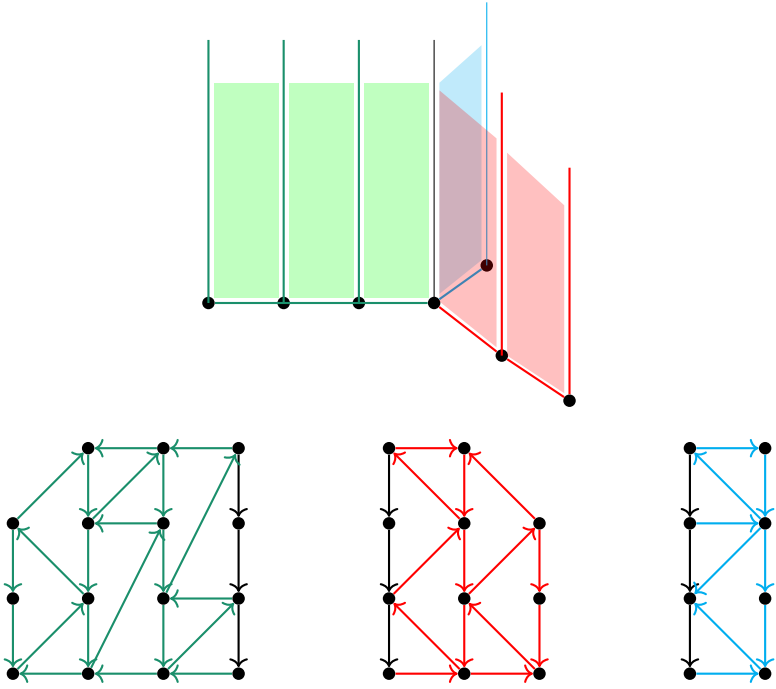
\begin{figure}[ht!]
\begin{tikzpicture}
\node [fill,circle,scale=0.5,label=below:$$] (1) at (0,0) {};
\node [fill,circle,scale=0.5,label=below:$$] (2) at (1,0) {};
\node [fill,circle,scale=0.5,label=below:$$] (3) at (2,0) {};
\node [fill,circle,scale=0.5,label=below:$$] (4) at (3.9,-0.7) {};
\node [fill,circle,scale=0.5,label=right:$$] (n) at (3.7,0.5) {};
\node [fill,circle,scale=0.5,label=below:$$] (3) at (3,0) {};

\node [fill,circle,scale=0.5,label=below:$$] (5) at (4.8,-1.3) {};
\draw [thick,-,jade] (1)-- (3);
\draw [thick,-,red] (3)-- (4);
\draw [thick,-,red] (4)-- (5);
\draw [thick,-,red] (5)-- (4.8,1.8);
\draw [thick,-,cyan] (3)-- (n);
\draw  [thick,-,jade] (0,0)-- (0,3.5);
\draw  [thick,-,jade] (1,0)-- (1,3.5);
\draw  [thick,-,jade] (2,0)-- (2,3.5);
\draw (3,0)-- (3,3.5);
\draw [cyan](3.7,0.5)-- (3.7,4);
\draw [thick,-,red] (3.9,-0.7)-- (3.9,2.8);
\fill[nearly transparent, green] (1.07,0.07) rectangle (1.93,2.92);
\fill[nearly transparent,green]  (2.93,2.92) rectangle (2.07,0.07);
\fill[nearly transparent, green] (0.07,0.07) rectangle (0.93,2.93);
\fill[nearly transparent,cyan] (3.07,0.12) -- (3.07,2.93)--(3.63,3.43)--(3.63,0.57);
\fill[nearly transparent,red] (3.07,0.02) -- (3.07,2.83)--(3.83,2.19)--(3.83,-0.59);
\fill[nearly transparent,red] (3.97,-0.7) -- (3.97,2)--(4.73,1.3)--(4.73,-1.2);
\end{tikzpicture}

\medskip

\begin{tikzpicture}
\node [fill,circle,scale=0.5,label=below:$$] (1) at (0,0) {};
\node [fill,circle,scale=0.5,label=below:$$] (2) at (1,0) {};
\node [fill,circle,scale=0.5,label=below:$$] (3) at (2,0) {};
\node [fill,circle,scale=0.5,label=below:$$] (4) at (3,0) {};
\node [fill,circle,scale=0.5,label=below:$$] (5) at (0,1) {};
\node [fill,circle,scale=0.5,label=below:$$] (6) at (1,1) {};
\node [fill,circle,scale=0.5,label=below:$$] (7) at (2,1) {};
\node [fill,circle,scale=0.5,label=below:$$] (8) at (0,2) {};
\node [fill,circle,scale=0.5,label=below:$$] (9) at (1,2) {};
\node [fill,circle,scale=0.5,label=below:$$] (10) at (2,2) {};
\node [fill,circle,scale=0.5,label=below:$$] (12) at (1,3) {};
\node [fill,circle,scale=0.5,label=below:$$] (13) at (2,3) {};
\node [fill,circle,scale=0.5,label=below:$$] (15) at (3,3) {};
\node [fill,circle,scale=0.5,label=below:$$] (16) at (3,2) {};
\node [fill,circle,scale=0.5,label=below:$$] (17) at (3,1) {};

\draw [thick,->, jade] (8)-- (5);
\draw [thick,->, jade] (5)-- (1);
\draw [thick,->, jade] (15)-- (13);
\draw [thick,->, jade] (9)-- (6);
\draw [thick,->, jade] (6)-- (2);
\draw [thick,->, jade] (12)-- (9);
\draw [thick,->, jade] (3)-- (2);
\draw [thick,->, jade] (13)-- (10);
\draw [thick,->, jade] (10)-- (7);
\draw [thick,->, jade] (7)-- (3);
\draw [thick,->, jade] (2)-- (1);
\draw [thick,->, jade] (4)-- (3);

\draw [thick,->] (15)-- (16);
\draw [thick,->] (16)-- (17);
\draw [thick,->] (17)-- (4);

\draw [thick,->, jade] (1)-- (6);
\draw [thick,->, jade] (6)-- (8);
\draw [thick,->, jade] (8)-- (12);
\draw [thick,->, jade] (13)-- (12);
\draw [thick,->, jade] (9)-- (13);
\draw [thick,->, jade] (10)-- (9);
\draw [thick,->, jade] (2)-- (10);
\draw [thick,->, jade] (3)-- (17);
\draw [thick,->, jade] (17)-- (7);
\draw [thick,->, jade] (7)-- (15);

\node [fill,circle,scale=0.5,label=below:$$] (18) at (5,0) {};
\node [fill,circle,scale=0.5,label=below:$$] (19) at (5,1) {};
\node [fill,circle,scale=0.5,label=below:$$] (20) at (5,2) {};
\node [fill,circle,scale=0.5,label=below:$$] (21) at (5,3) {};
\node [fill,circle,scale=0.5,label=below:$$] (22) at (6,0) {};
\node [fill,circle,scale=0.5,label=below:$$] (23) at (6,1) {};
\node [fill,circle,scale=0.5,label=below:$$] (24) at (6,2) {};
\node [fill,circle,scale=0.5,label=below:$$] (25) at (6,3) {};
\node [fill,circle,scale=0.5,label=below:$$] (26) at (7,0) {};
\node [fill,circle,scale=0.5,label=below:$$] (27) at (7,1) {};
\node [fill,circle,scale=0.5,label=below:$$] (28) at (7,2) {};

\draw [thick,->] (21)-- (20);
\draw [thick,->] (20)-- (19);
\draw [thick,->] (19)-- (18);
\draw [thick,->, red] (25)-- (24);
\draw [thick,->, red] (24)-- (23);
\draw [thick,->, red] (23)-- (22);
\draw [thick,->, red] (28)-- (27);
\draw [thick,->, red] (28)-- (25);
\draw [thick,->, red] (27)-- (26);
\draw [thick,->, red] (21)-- (25);

\draw [thick,->, red] (23)-- (28);
\draw [thick,->, red] (24)-- (21);
\draw [thick,->, red] (19)-- (24);
\draw [thick,->, red] (18)-- (22);
\draw [thick,->, red] (22)-- (26);
\draw [thick,->, red] (26)-- (23);
\draw [thick,->, red] (22)-- (19);

\node [fill,circle,scale=0.5,label=below:$$] (30) at (9,0) {};
\node [fill,circle,scale=0.5,label=below:$$] (31) at (9,1) {};
\node [fill,circle,scale=0.5,label=below:$$] (32) at (9,2) {};
\node [fill,circle,scale=0.5,label=below:$$] (33) at (9,3) {};
\node [fill,circle,scale=0.5,label=below:$$] (34) at (10,0) {};
\node [fill,circle,scale=0.5,label=below:$$] (35) at (10,1) {};
\node [fill,circle,scale=0.5,label=below:$$] (36) at (10,2) {};
\node [fill,circle,scale=0.5,label=below:$$] (37) at (10,3) {};

\draw [thick,->] (33)-- (32);
\draw [thick,->] (32)-- (31);
\draw [thick,->] (31)-- (30);
\draw [thick,->, cyan] (32)-- (36);
\draw [thick,->, cyan] (37)-- (36);
\draw [thick,->, cyan] (36)-- (35);
\draw [thick,->, cyan] (35)-- (34);
\draw [thick,->, cyan] (30)-- (34);
\draw [thick,->, cyan] (34)-- (31);
\draw [thick,->, cyan] (36)-- (31);
\draw [thick,->, cyan] (36)-- (33);
\draw [thick,->, cyan] (33)-- (37);

\end{tikzpicture}
 \caption{An example of a dimer model on the cylinder over $E_7$.  The quivers lie on their respective colored sheets, and they share the three black arrows where the sheets intersect. }
\label{fig:e7dimer}
\end{figure}

\begin{ex} 
Quivers for double Bruhat cells of $E_7$ can be drawn on a book-like structure as shown in ~Figure \ref{fig:e7dimer}. The cylinder $E_7 \times \RR$ has seven strings and three sheets: one sheet of length 3 (green in color), one sheet of length 2(red in color) and one sheet of length 1 (blue in color) glued together at their boundaries (the black string). 

\end{ex}

\subsection{Double Bruhat cells} 

Let $G$ be a Kac--Moody group. Let $B$ and $B_{-}$ be opposite borel subgroups and $W$ be its Weyl group. Then $G$ can be written as disjoint union of double Bruhat cells $G_{u,v}$ where, $G_{u,v}=BuB\cap B_{-}vB_{-}$ and $u,v \in W$.  To each such pair of Weyl elements $(u, v)$, we can associate a quiver $Q^{u,v}$ as defined in \cite{bfz05}.

\subsection{Berenstein--Fomin--Zelevinsky quivers}
Let $G$ be a simply connected complex algebraic group. Let $W$ be the Weyl group and $\mathfrak{g}$ be the Lie algebra of $G$. Every Weyl group can be realized as a Coxeter group with reflections ${s_1, s_2, \ldots, s_r}$ of simple roots as its generators. Each $s_i$ is an involution and $(s_is_j)^{m_{ij}} = 1 \in W$, for some integer $m_{ij}$ encoded in the Dynkin diagram.  Every element $w \in W$ has a smallest expression in terms of $s_i$'s.  A word is a tuple of indices of simple reflections in the smallest expression for $w$.   If for $w=s_{i_1}s_{i_2}\cdots s_{i_l}$ in $W$ is the smallest such expression in terms of the generators of $W$ then the word $i =(i_1, i_2, \cdots, i_l), \  i_j\in[1, \cdots, r]$ is said to be in its reduced form.  The length of the word $w$ is denoted by $\ell(w)$ and $\ell(w)= l$.

Fix a pair $(u,v) \in W \times W$. Let us use negative indices for the generators of the first copy of $W$ and positive indices for the second copy of $W$.  Then a reduced word $\textbf{i} = (i_1, \ldots, i_{\ell(u)+\ell(v)})$ is an arbitrary shuffle of a reduced word for $u$ and a reduced word for $v$. Let $\hat{\textbf{i}}=(-r, \ldots, -1, i_1, \ldots,  i_{\ell(u)+\ell(v)})$. For $k \in [-r, -1] \cup [1, \ell(u)+\ell(v)]$, $k^+$ is the smallest index $l$ such that $k<l$ and $|i_k|=|i_l|$.  If $|i_k|\neq |i_l|$ for any $l>k$, then $k^+= \ell(u)+\ell(v)+1$. An index is called $i$-exchangeable if both $k$ and $k^+$ are in $[-r, -1] \cup [1, \ell(u)+\ell(v)]$.

\begin{defn}\label{def:bfz} Let $u, v \in W$.  A BFZ quiver $Q^{u,v}$ has set of vertices $Q_0 = \hat{\textbf{i}}$  Vertices $k$ and $l$ such that $k < l$ are connected if and only if either $k$ or $l$ are $i$-exchangeable. There are two types of edges: 
\begin{itemize}
\item  An edge is called horizontal if $l = k^+$ and it is directed from $k$ to $l$ if and only if $\epsilon(i_l) =+1$.
\item An edge is called inclined if one of the following conditions hold: 
\begin{enumerate}
\item $l < k^+ < l^+$, $a_{|i_k|, |i_j|} > 0$, $\epsilon(i_l) = \epsilon(i_{k^+})$
\item $l < l^+ < k^+$, $a_{|i_k|, |i_j|} > 0$, $\epsilon(i_l) = -\epsilon(i_{l^+})$
\end{enumerate} 
An inclined edge is directed from $k$ to $l$ if and only if $\epsilon(i_l) =-1$.
\end{itemize}
\end{defn}
These quivers are used in \cite{bfz05} to produce a cluster structure on the coordinate rings of double Bruhat cells.
\begin{theorem} A BFZ quiver can be realized as a dimer model on the cylinder over the corresponding Dynkin diagram.
\end{theorem}
\begin{proof} Notice that the horizontal edges in ~Definition \ref{def:bfz} lie on the strings of the cylinder over the Dynkin diagram.  All inclined edges lie between two adjacent strings such that they project down onto an edge of the Dynkin diagram.   According to the definition of the quiver, there is an edge between two vertices of adjacent strings only if the corresponding vertices in the graph are connected by an edge. 
\end{proof}

\subsection{BFZ quivers for type $A_n$}
A quiver for double Bruhat cells for $A_n$ can be viewed as a quiver on a plane on $A_n$ as shown below. All vertical edges in the quiver project onto vertices in the Dynkin diagram. All inclined edges project onto edges of the Dynkin diagram. 

\begin{ex}
Consider the Lie group of type $A_3$. The Weyl group in this case is $W=S_4$. Let $u=s_3s_2s_1s_2s_3, v=e \in S_4$. The quiver $Q^{u,v}$ corresponding to the double Bruhat cell $B_{u,v}$ is as shown below:

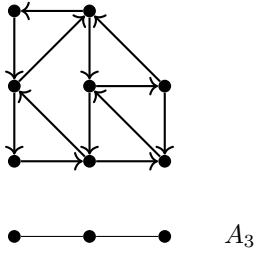
\begin{figure}[ht!]
\begin{tikzpicture}
\node [fill,circle,scale=0.5] (v2) at (0,0) {};
\node [fill,circle,scale=0.5] (v3) at (1,0) {};
\node [fill,circle,scale=0.5] (v4) at (2,0) {};
\node [fill,circle,scale=0.5] (v5) at (0,1) {};
\node [fill,circle,scale=0.5] (vn) at (0,2) {};
\node [fill,circle,scale=0.5] (v1) at (1,1) {};
\node [fill,circle,scale=0.5] (v6) at (1,2) {};
\node [fill,circle,scale=0.5] (v) at (2,1) {};
\node [fill,circle,scale=0.5] (1) at (0,-1) {};
\node [fill,circle,scale=0.5]  (2)at (1,-1) {};
\node [fill,circle,scale=0.5] (3) at (2,-1) {};
\node at (3,-1) {$A_3$};
\draw (1) -- (2);
\draw (2) -- (3);
\draw [thick,->] (v5)-- (v2);
\draw [thick,->] (v2)-- (v3);
\draw [thick,->] (v1)-- (v3);
\draw [thick,->] (v3) -- (v4) ;
\draw [thick,->] (vn)-- (v5);
\draw [thick,->](v6) -- (v1) ;
\draw [thick,->](v) -- (v4) ;
\draw [thick,->] (v3)-- (v5);
\draw [thick,->] (v5)-- (v6);
\draw [thick,->] (v6)-- (vn);
\draw [thick,->] (v)-- (v6);
\draw [thick,->] (v1)-- (v);
\draw [thick,->] (v4)-- (v1);
\end{tikzpicture}
\end{figure}

\end{ex}

\begin{lemma} A BFZ quiver $Q^{u,v}$ is planar in each sheet. 
\end{lemma}
\begin{proof}
Consider the $k$th and the $l$th string of the quiver.  If the strings are not adjacent on a sheet, then we know that there cannot be edges between the vertices of the strings.  If the strings are adjacent, consider the following diagram:
\begin{center}
$\begin{tikzcd}
k \ar{r} & k^{+} \ar{dl}\ar{r} & k^{++} \ar{r} & \cdots \ar{r} & k^{p+}\\ 
l \ar{r}  & l^{+} \ar{r} & l^{++} \ar{r} & \cdots \ar{r} & l^{r+} 
\end{tikzcd}$
\end{center}

Suppose the vertices $l$ and $k^{+}$ are connected.  Then depending on whether $k^{+}< l$ or $l< k^{+}$, there will be the following inequalities:\\
(1) If  $k^{+}< l$, then  $ l<k^{++}<l^{+}$ \\
(2) If  $l< k^{+}$, then $k^{+}< l^{+}<k^{++}$. \\
(3) So in both cases above, $k^+ < l^+$. \\
 We want to show that the vertex $k$ is not connected to $l^{m+}$ for any $m$. Suppose $k$ and $l^{m+}$ are connected.  Then again, there are two cases: \\
(4) If  $l^{m+}<k $, then  $ k<l^{(m+1)+}<k^{+}$ \\
(5) If  $k< l^{m+}$, then $l^{m+}< k^{+}<l^{(m+1)+}$. \\
(6) Combining inequalities in (4) and (5) with $l^+<l^{m+}$ we get, $l^+ <k^+$. \\
As (3) and (6) contradict each other, there cannot be an overlapping edge.  Hence the quiver is planar in each sheet.
\end{proof}

\begin{lemma} For any Kac--Moody algebra $\mathfrak{g}$ and $(u,e) \in W \times W$, all faces of $Q^{u,e}$ are oriented, where $e$ is the identity element in $W$. 
\end{lemma} 
\begin{proof} Let us assume that there exists a non-oriented $n$- cycle in the quiver with $p+1$ vertices in $j$th string, $r+1$ vertices in (a neighboring) $k$th string and $n=p+r+2$. Note that all edges in all strings are directed in one direction as we are fixing one of the Weyl group elements to be the identity. Two edges between the neighboring strings can be directed in the same or opposite direction (as shown below).  Let us consider the first case where the vertical edges have the same direction. 
\begin{center}
$\begin{tikzcd}
j \ar{r} & j^{+} \ar{r} & j^{++} \ar{r} & \cdots \ar{r} & j^{p+}\\ 
k \ar{r} \ar{u} & k^{+} \ar{r} & k^{++} \ar{r} & \cdots \ar{r} & k^{r+} \ar{u}
\end{tikzcd}$
\end{center}

Let $r\leq p$. From the direction of the vertical arrows, it is clear that $j<k<j^+<k^+$ and $j^{p+}<k^{r+}<j^{(p+1)+}<k^{(r+1)+}$. We also know that $j<j^+<\cdots <j^{p+}<j^{(p+1)+}$ and $k<k^+<\cdots <k^{r+}<k^{(r+1)+}$.

Each inequality  $j^{m+}<\cdots<j^{(m+s)+}<k^{n+}< \cdots<k^{(n+t)+}<j^{(m+s+1)}$ creates an edge from $j^{(m+s)+}$ to $j$. For every such inequality, notice that we get one or more edges in the $n$-cycle which divides the cycle into smaller oriented cycles.  The second case where the two edges between the neighboring strings have opposite directions follows similarly from corresponding inequalities. \scshape
\end{proof}

\begin{remark} Each $n$-face in a quiver has $n-1$ vertices in one string and the remaining one vertex in its adjacent string.
\begin{center}
$\begin{tikzcd}
j \ar{r} & j^{+} \ar{r} & j^{++} \ar{r} & \cdots \ar{r} & j^{p+} \ar{dll}\\ 
 &  & k \ar{ull} 
\end{tikzcd}$
\end{center} 
\end{remark} 

\section{Rigidity of the superpotential}
In this section, we will use the planarity of a dimer models on a cylinder to show that its superpotential is rigid, in certain cases.  Recall that a potential of a quiver is a linear combination of cycles in the quiver.  The potential  \[S= \sum{\text{clockwise oriented faces}}-\sum{\text{anti-clockwise oriented faces}}\] is called the superpotential of the quiver $Q$. 

\begin{remark} Each vertex of a quiver $Q^{u,e}$ has at most one edge going to and at most one edge coming from each adjacent string.
\end{remark} 
As an application of the theory of dimer models on cylinders we give an independent proof of the following result of \cite{birs09}:
\begin{theorem}The superpotential $S$ of the quiver $Q^{u,e}$ is a rigid potential. 
\end{theorem}

We first prove that the sub-potential of the superpotential $S$ lying in each sheet is rigid. Recall that two sheets are glued at a string, hence the sub-potentials share edges between them, but they do not share any faces of the quiver. Therefore, rigidity of the sub-potentials indeed implies rigidity of the superpotential. Denote by $S_r$, the sub-potential of the superpotential $S$ that lies on the $r$ th sheet. In order to prove rigidity of $S_r$, we need to show that each cycle in the quiver belongs to the Jacobian ideal $J(S_r)$. 

\begin{lemma} Every face belongs to $J(S_r)$.
\end{lemma}
\begin{proof} We will prove this by induction on the distance of a face from the boundary of the quiver. The distance of a face $F$ is denoted by $d(F)$ and defined as follows: $d(F)=0$ if $F$ has a boundary edge as one of its edges. If $F$ is not a boundary face, then $d(F)=d + 1,$ where $d= \min\{{d(F^{\prime}) | \\\ F^{\prime}\text{ is an adjacent face to $F$}}\}$.

Now, if $d(F)=0$, $F$ has a boundary edge as one of its edges. Let us call the edge $f$, then $F=f\partial_f(S_r)\in J(S_r)$, which implies all boundary faces are in the Jacobian ideal. Suppose $d(F)=n+1$, then $F$ has at least one adjacent face whose distance is $n$. Let that face be $E$ and $e$ be the edge shared by $E$ and $F$. As $d(E)=n$, by induction, $E \in J(S_r)$, and by the definition of the Jacobian ideal, $e\partial_e(S_r)=E+F \in J(S)$, therefore $F \in J(S_r)$.  Hence all faces with distance $n+1$ are in the Jacobian ideal, which completes the proof by induction.
\end{proof}
Note that if a cycle is self-intersecting, it can be written a product of two or more non-self-intersecting cycles.  If we want to show that the original cycle belongs to the Jacobian ideal, then it suffices to prove that one of its non-self-intersecting cycles belongs to the ideal. 

\begin{defn} A cycle $C$ is called differentiable with respect to an edge $e$ if $e$ separates the cycle into a face and a smaller cycle. 
\end{defn}
As the quiver is planar in each sheet, we know  that the edge $e$ is shared by at most two faces, say $F_1$ and $F_2$. If $C$ is differentiable with respect to the edge $e$, then $C$ contains all edges of either $F_1$ or $F_2$ except $e$.
\begin{lemma} Every non-self-intersecting cycle in the quiver is differentiable with respect to at least one edge in its interior. 
\end{lemma}
\begin{proof} Let $C$ be a cycle containing $k$ faces, $F_1, F_2, \ldots, F_k$. Suppose $F_i$ has $n_i$ vertices and $n_i$ edges. We need to show that $C$ contains all but one edge of $F_j$ for some $j$. 

Recall that each face $F_i$ has one of its vertices in a string and the remaining $n_i-1$ vertices in its adjacent string. Each vertex has at most one edge going to and at most one edge coming from each adjacent string. Lastly, every edge in each string is directed in the same direction. 

Let $p:v_2\to v_n$ be the right-most vertical path in cycle $C$. Suppose $p$ belongs to the $k$th string, $r_k$ of the quiver. This path $p$ has exactly one inclined edge $e_1:v_1\to v_2$ from the string to its left, $r_{k-1}$. That means, the face that contains $p$ and $e$, has its vertex $v_1$ in string $r_k$ and all remaining vertices in $r_{k-1}$ that belong to path $p$. As $p$ is the right-most vertical path of $C$, the edge $e_n:v_n\to v_{n+1}$ lands in the string $r_{k-1}$. If $v_1= v_{n+1}$, the cycle is self-intersecting. Hence $v_1\neq v_{n+1}$. So, for some $2<m<n$, there is an edge $e_m:v_m \to v_1$,  which lies in the interior of $C$ and completes a face in the quiver. This edge $e_m$ separates $C$ into a face (consisting of  vertices $v_1, v_2, \ldots, v_m$) and a smaller cycle, and hence $C$ is differentiable with respect to $e_m$.

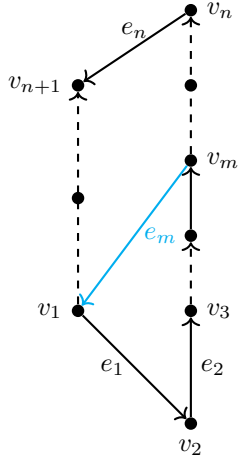
\begin{center}
\begin{tikzpicture}[rotate=90]
\node [fill,circle,scale=0.5,label=below:$v_2$] (v2) at (-0.5,0) {};
\node [fill,circle,scale=0.5,label=right:$v_3$] (v3) at (1,0) {};
\node [fill,circle,scale=0.5] (v4) at (2,0) {};
\node [fill,circle,scale=0.5,label=right:$v_m$] (vm) at (3,0) {};
\node [fill,circle,scale=0.5] (v5) at (4,0) {};
\node [fill,circle,scale=0.5,label=right:$v_n$] (vn) at (5,0) {};
\node [fill,circle,scale=0.5,label=left:$v_1$] (v1) at (1,1.5) {};
\node [fill,circle,scale=0.5] (v6) at (2.5,1.5) {};
\node [fill,circle,scale=0.5,label=left:$v_{n+1}$] (v) at (4,1.5) {};
\draw [thick,->] (v2)-- node[right]{$e_2$} (v3);
\draw [thick,dashed, ->](v3) -- (v4) ;
\draw [thick,->] (v4)-- (vm);
\draw [thick,dashed, ->](vm) -- (vn) ;
\draw [thick,dashed, ->](v1) -- (v) ;
\draw [thick,->] (v1)-- node[left]{$e_1$} (v2);
\draw [thick,->,cyan] (vm)-- node[right]{$e_m$} (v1);
\draw [thick,->] (vn)-- node[above]{$e_n$} (v);
\end{tikzpicture}
\end{center}
\end{proof} 

\begin{lemma} Any non-self-intersecting cycle $C$ can be written as multiplication of a face and a cycle in the Jacobian algebra.  
\end{lemma}
\begin{proof} We use induction on $k$, the number of faces contained inside the cycle.  If a cycle contains only one face, then by the lemma above, it belongs to the Jacobian ideal.

Let $C$ be a cycle containing $k$ faces, $F_1, F_2, \ldots, F_k$, with $n_1, n_2, \ldots, n_k$ number of edges respectively. We know that $C$ contains all vertices of at least one of these $k$ faces. Let that face be $F_i$, which starts and ends at the vertex $e_1$. So $C$ contains $n_i-1$ edges of $F_i$. Let $p_1:e_1\rightarrow e_{n_i}$ be the path consisting of $n_i-1$ edges of $F_i$ that also belong to $C$. As $C$ is a cycle, there exists a path, say $p_2:e_{n_i}\rightarrow e_1$ such that $C=p_1p_2$. Let $e$ be the $n_i$th edge of the face $F_i$ such that $F_i=e p_1.$ Now there exists a path $p_1^{\prime}:e_1\to e_{n_i}$ such that $\partial_e(S)=p_1-p_1^{\prime}$, which implies that $p_1=p_1^{\prime}$ in the Jacobian algebra. Hence $C=p_1p_2=p_1^{\prime}p_2$, reducing the number of faces inside $C$ to $k-2$.
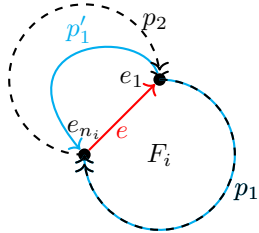
\begin{center}
\begin{tikzpicture}
\node [fill,circle,scale=0.5,label=above:$e_{n_i}$] (en) at (0,-1) {};
\node [fill,circle,scale=0.5,label=left:$e_1$] (e1) at (1,0) {};
\draw [thick,red, ->] (en)-- node[below]{$e$} (e1);
\draw [thick,cyan,->](e1) .. node[above]{$p_1^{\prime}$} controls (0.75,0.75) and (-1.25,0.5) .. (en) ;
\draw [thick,dashed, ->>](en) arc [radius=1, start angle=270, end angle=50] node[right]{$p_2$}arc [radius=1, start angle=50, end angle=0] ;
\draw [thick, ->>, cyan](e1) arc [radius=1, start angle=90, end angle=-30] node[right]{$p_1$} arc [radius=1, start angle=-30, end angle=-180];
\draw [thick, ->>, dashed](e1) arc [radius=1, start angle=90, end angle=-30] node[right]{$p_1$} arc [radius=1, start angle=-30, end angle=-180];
\node at (1,-1) {$F_i$};
\end{tikzpicture}
\end{center}
\end{proof}

\noindent This shows that every cycle of a quiver $Q^{u,e}$ belongs to the Jacobian ideal corresponding to the superpotential $S$. Hence $S$ is a rigid potential. 


\bibliographystyle{alpha}

\end{document}